\documentclass[11pt,a4paper,twoside]{amsart}
\usepackage[latin1]{inputenc}
\usepackage[T1]{fontenc}

\usepackage{comment}
\usepackage{datetime2}
\usepackage{datetime2-calc}
\usepackage{tikz-cd}
\usepackage{color, mathtools}
\usepackage{amssymb}
\usepackage{microtype}
\usepackage{graphicx}
\usepackage{verbatim}
\usepackage{hyperref}
\usepackage{booktabs}

\theoremstyle{plain}
\newtheorem{thm}{Theorem}[subsection]
\newtheorem{lem}[thm]{Lemma}
\newtheorem{prop}[thm]{Proposition}

\newtheorem{cor}[thm]{Corollary}
\newtheorem{conj}[thm]{Conjecture}

\theoremstyle{definition}
\newtheorem{df}[thm]{Definition}
\theoremstyle{remark}
\newtheorem{ex}[thm]{Example}

\newtheorem{rem}[thm]{Remark}
\newtheorem{quest}[thm]{Question}

\DeclareMathOperator{\Gal}{Gal}
\DeclareMathOperator{\borel}{Bo}
\DeclareMathOperator{\bloch}{Bl}
\DeclareMathOperator{\rank}{rk}
\DeclareMathOperator{\lcm}{lcm}
\DeclareMathOperator{\tor}{tors}
\newcommand{\tors}{\tor}
\DeclareMathOperator{\Hom}{Hom}
\DeclareMathOperator{\reg}{reg}

\newcommand{\ZZ}{\mathbb{Z}}
\newcommand{\Z}{\ZZ}
\newcommand{\PP}{\mathbb{P}}
\newcommand{\QQ}{\mathbb{Q}}
\newcommand{\Q}{\QQ}

\newcommand{\RR}{\mathbb{R}}
\newcommand{\R}{\RR}

\DeclarePairedDelimiterX\set[1]{\lbrace}{\rbrace}{\def\given{\;\delimsize\vert\;}#1}
\DeclarePairedDelimiterX\gen[1]{\langle}{\rangle}{\def\given{\;\delimsize\vert\;}#1}
\DeclarePairedDelimiterX\abs[1]{\lvert}{\rvert}{#1}
\DeclarePairedDelimiterX\size[1]{\lvert}{\rvert}{#1}
\DeclarePairedDelimiterX\norm[1]{\lVert}{\rVert}{#1}
\DeclarePairedDelimiterX\symb[1]{[}{]}{#1}
\newcommand{\wtilde}[2]{#1 \mathbin{\tilde{\wedge}} #2 }

\newcommand{\wB}{\mathcal{B}}

\newcommand{\XX}{\mathbf{X}}
\newcommand{\OO}{\mathcal{O}}
\newcommand{\C} {\mathbb{C}}

\newcommand{\CL} {\mathcal{L}}

\newcommand{\B} {\mathcal B}

\newcommand{\tCL}{\tilde{\CL}}
\newcommand{\tnu}{\tilde{\nu}}

\newcommand{\tA}{\tilde{A}}
\newcommand{\tB}{\tilde{B}}
\newcommand{\tC}{\tilde{C}}
\newcommand{\td}{\tilde{\delta}}
\newcommand{\twedgetwo}{\tilde{\bigwedge}{}^{\hskip-2pt 2 \hskip 2pt}}
\newcommand{\bwedgetwo}{\bigwedge{}^{\hskip-2pt 2 \hskip 2pt}}
\newcommand{\ceq}{\stackrel{?}{=}}

\def \newline{{\par\noindent}}

\def \Li{{\rm Li}}

\def \CO{{\mathcal O}}

\begin{document}
\date{October 16, 2026}
\title{Computations in weight 3 of higher Bloch groups}

\begin{abstract}
We computationally investigate a weight~$3$ analogue of the
relationship between higher Bloch groups, algebraic $K$-theory,
higher regulators, and special values of Dedekind zeta functions.  In particular, for $F$ a number field $F$, we define the
\emph{pruned Bloch group} $\B _3(F)$ and compare the covolume of its regulator lattice with $\zeta_F(3)$. We define a quantity $\kappa_3(F)$, which is conjecturally an integer, that is expected to be closely
related to the order of $K_4(\mathcal{O}_F)$. 
\end{abstract}

\author[H. Gangl]{Herbert Gangl}
\address{H. Gangl, Department of Mathematical Sciences, Stockton Road, Durham DH1 3LE, United Kingdom}
\email{herbert.gangl@durham.ac.uk}
\urladdr{\url{http://maths.dur.ac.uk/~dma0hg/}}

\author[P. E. Gunnells]{Paul E. Gunnells}
\address{P. E. Gunnells, Department of Mathematics and Statistics, LGRT 1115L, University of Massachusetts, Amherst, MA 01003, USA}
\email{gunnells@umass.edu}
\urladdr{\url{https://gunnells.github.io}}

\author[D. Yasaki]{Dan Yasaki} \address{D. Yasaki, Department of
Mathematics and Statistics, University of North Carolina at
Greensboro, Greensboro, NC 27412, USA}
\email{d\_yasaki@uncg.edu}
\urladdr{\url{https://go.uncg.edu/d_yasaki}}

\thanks{This work was partially supported by grants 963025 (PG) and
848154 (DY) from the Simons Foundation. HG is grateful to the MPI Bonn
for its hospitality and for providing perfect working conditions.}

\maketitle
\setcounter{tocdepth}{1}

\section{Introduction}

\subsection{}
Let $F$ be a number field with $r_1$ real and $r_2$ complex places, discriminant $\Delta_F$, and ring of integers $\CO=\CO_F$.  The celebrated Borel Theorems \cite{BorelZeta} give 
the rank of its algebraic $K$-groups $K_{2m-1}(F)$, $m>1$, as 
\[d_m \coloneq \rank K_{2m-1}(F) = \begin{cases} r_2,&\text{ if }m\geq2\text{ even,}\\ r_1+r_2 &\text{ if } m\geq 3\text{ odd.}\end{cases}\]
In the same paper, Borel also shows that  the even-indexed groups $K_{2m}(F)$ for $m>0$ have rank $0$.

A crucial ingredient in Borel's proof is the construction for $m>1$ of a higher \emph{regulator map} sending $K_{2m-1}(\C)$ to $\R$.  Composing this map with $d_m$ suitable complex embeddings $F\hookrightarrow\C$ yields a mapping $K_{2m-1}(F)$ to a full rank lattice $\Lambda^{\borel}_m(F)$ in $\R^{d_m}$.
Moreover, he shows that up to a known power of $\pi$ and
$\sqrt{|\Delta_F|}$, the covolume of $\Lambda^{\borel}_m(F)$ is a
non-zero rational multiple of the special value at $s=m$ of the
Dedekind zeta function $\zeta_F(s)$.  This covolume is the \emph{Borel
regulator} $\reg_m^{\borel}(F)$.  Note that, according to Soul\'e
\cite[Th\'eor\`eme 3]{SouleInventiones}, in conjunction with Quillen's
vanishing result for $K_{2n}(k)$ for a finite field $k$ and $n>0$, we
have $K_{2m-1}(\CO_F) = K_{2m-1}(F)$ for $m>1$ (cf.~Theorem~\ref{thm:Kodd} below).

\subsection{} It turns out to be very unwieldy to try to explicitly compute the Borel regulator on elements in $K_{2m-1}(\CO_F)$. Already for the simplest case of $m=2$ and $F=\Q(\sqrt{-3})$ substantial computational efforts did not allow the authors of \cite{ChooGarciaSnaith} to approximate that regulator very well.  In his seminal Irvine Lecture notes \cite{BlochIrvine},
Bloch envisaged the construction of a group $B_2(F)$, in terms of
infinitely many generators and relations, that he conjectured to be
closely related to $K_3(F)$. Later Suslin proved Bloch's vision to be correct.  He gave a precise relationship between the groups, stating that there is a surjection from $K_3(F)$  onto $B_2(F)$ with a specific torsion group as its kernel \cite[Theorem~5.2]{SuslinK3}.
 Moreover, Bloch concocted an $\R$-valued map
$\CL_2$, now known as the \emph{Bloch-Wigner dilogarithm}, that
can be used to map $B_2(F)$ to  $\R^{r_2}$. He
expected the dilogarithm $\CL_2$ to play a similar role for $B_2(F)$ as the Borel
regulator map does for $K_3(\OO_F)$. In particular, he conjectured $\CL_2$ maps $B_2(F)$
to a lattice in $\R^{r_2}$
with covolume rationally proportional to $\zeta_F(2)$
(again up to a known normalizing factor). Furthermore he corroborated this for 
cyclotomic fields \cite[Theorem 7.2.4]{BlochIrvine} and imaginary quadratic cases \cite{BlochApplications}.
This relationship with $\zeta_F(2)$ was proved, in a slightly weaker form, by Zagier \cite{ZagierHyperbolicVolumes}, and is also a consequence of the cited works by Borel, Bloch and Suslin.  See also Ramakrishnan \cite{RamakrishnanSurvey} and Goncharov \cite{Goncharov95}.

For general $K_{2m-1}(F)$, $m\geq 2$, Zagier \cite[p.~416]{ZagierTexel} proposed a conjectural generalization $B_m(F)$ of Bloch's $B_2(F)$.  The construction again uses infinitely many generators and relations, and as before the group $B_m(F)$ is expected to be closely related to $K_{2m-1}(\OO_F)$. More specifically, 
single-valued higher analogues of $\CL_2(z)$ were defined by Ramakrishnan \cite{Ramakrishnan} and shown to satisfy a functional equation.  Zagier later expressed  Ramakrishnan's function explicitly in terms of classical polylogarithms as a function $\CL_m(z)$ \cite{ZagierBlochWignerRamakrishnan}.(\footnote{Zagier in fact gave three different explicit realizations, denoted in \cite{ZagierBlochWignerRamakrishnan} by $D_m(z)$ (eq.~12), $\widetilde{D}_m(z)$ (eq.~29) and finally $P_m(z)$ (eq.~33).  The first version was also given simultaneously by Wojtkowiak \cite{Wojtkowiak}; the third version later appears as $D_m(z)$ (up to a possible factor $\sqrt{-1}$) in Beilinson--Deligne \cite{BeilinsonDeligne} and is the most widespread version.  $P_m(z)$ is the function we denote $\CL_m(z)$.}) 
Zagier applied a map built from $\CL_m(z)$ to $B_m(F)$ to produce a lattice in $\R^{d_m}$. Zagier's conjecture \cite[p.~417]{ZagierTexel} states that there is a canonical map
\begin{equation}\label{eq:dbz}
    B_m(F) \longrightarrow K_{2m-1}(\OO_F)
\end{equation}
with finite kernel and finite cokernel whose composite with the Borel regulator map coincides with the image of $B_m(F)$ under the polylogarithm map.  Using this construction he numerically verified the expected relation of the covolume of the image with the special value $\zeta_F(m)$ for some fields $F$. 
Beilinson and Deligne \cite{BeilinsonDeligne} give a map as in \eqref{eq:dbz}, but without proof of the finiteness statements and conditionally on the existence of a category of mixed Tate motives over a number field.  The existence was subsequently proved by Deligne and Goncharov \cite{DeligneGoncharov}.  In  \cite{deJeu1995Zagier}, de Jeu unconditionally constructs \eqref{eq:dbz} a rational injection of his multi-relative $K$-theory variant of $B_m(F)$, which is moreover compatible with regulators.

\subsection{}
Returning now to $K$-theory, a beautiful insight of Lichtenbaum \cite{Lichtenbaum} provided a conjectural interpretation of the unknown rational factor in Borel's result connecting $\reg_m^{\borel}(F)$ to the special value $\zeta_F(m)$.  Namely he predicted a value for this rational factor in terms of the order of the even-indexed $K$-group $K_{2m-2}(\CO_F)$ and the order of the torsion subgroup $K_{2m-1}(\CO_F)_{\tors}$ of the odd-indexed $K$-group $K_{2m-1}(\CO_F)$.  The resulting formula is in the spirit of Dirichlet's Class Number Formula, which arises for $m=1$ when one computes the residue of the pole of $\zeta_F(s)$ at $s=1$.

One obtains a simpler formulation of the resulting conjecture \cite[Question~4.2]{Lichtenbaum} after applying Hecke's functional equation for $\zeta_F(s)$.  Up to a normalizing power of $\pi$ and powers of 2, one should have 
\begin{equation}\label{eq:licht}
    \zeta_F^*(1-m) \ceq \frac{\size{K_{2m-2}(\CO_F)}}{\size{K_{2m-1}(\CO_F)_{\tors}}}\reg_m^{\borel}(F),
\end{equation}
where $\zeta^*_F(m)$ denotes the first non-vanishing term in the Taylor series expansion of $\zeta_F(s)$ at $s=m$.  Note that in \eqref{eq:licht} and in the sequel, we denote a conjectural identity $\ceq $.

The conjecture \eqref{eq:licht} has been proved, at least away from the $2$-part, in the special case of an abelian field $F$.  Kolster--Nguyen--Fleckinger~published a proof in \cite{Kolsteretal} but with erroneous Euler factors due to conflation of two lattices.  This error was pointed out by the first author to them, and subsequently the Euler factors were corrected by Benois--Nguyen \cite[Appendix A3]{BenoisNguyen}. The same result was subsequently reproved by Burns--Greither  \cite{BurnsGreither} and Huber--Kings \cite{HuberKings}. Important later work by Flach \cite{Flach} completed the result by treating the $2$-part; see also \cite{k3gens} for a discussion of Flach's work in the context of $K_3(\OO_F)$ and $B_2(F)$.

\subsection{}
The aim of this paper is to combine the framework of Zagier's
conjecture with Lichtenbaum's ideas, and in particular to numerically
investigate this circle of ideas in the case of weight $m=3$.  The case of weight $m = 2$ has been investigated in earlier work \cite{browkin-gangl, k3gens, guo-qin}.  Thus for a large set of number fields $F$ of varying signatures and Galois
closures we compute explicit examples of weight $3$ regulators and compare them 
with the special values $\zeta_{F} (3)$ and $\zeta^{*}_{F} (-2)$ with
the hopes of observing a precise analogue of Lichtenbaum's
\eqref{eq:licht}.  

For this purpose we make some modifications to the ingredients in \eqref{eq:licht}.  First, we replace the order of the torsion subgroup $\size{K_5(\OO_F)_{\tors}}$ in
\eqref{eq:licht} by an
integer $w_3(F)$ that agrees with the order of this subgroup up to a
possible factor of $2$ and is much simpler to compute; see
Corollary \ref{w3F} for a precise statement. Furthermore, instead of
considering $K_5(\OO_F)$, which is difficult to compute in practice, we
use a ``pruned'' version $\wB_3(F)$ of Zagier's third Bloch group
$B_3(F)$ (Definition \ref{def:working-bloch-group3}).  This notion lets us ignore any potential issues that are inherited from the loss of information by ignoring the $\R/(2\pi i)^2\Z$-part of the enhanced dilogarithm \cite{zickert-extended, zickert-holomorphic} and ``enshrines'' the standard practice of checking whether an element in $B_2(F)$ is trivial mod torsion by checking its vanishing under the dilogarithm.  We call the
group pruned because of this loss of information.  In order to emphasize the
parallel formulation,  we denote the covolume of the lattice generated
by the single-valued trilogarithm  $\CL_3$ on elements in $\wB_3(F)$
by $\reg_3^{\bloch}(F)$.  This covolume could be thought of as a
``regulator for the third Bloch group of $F$,'' although we do not expect it to have certain desired functorial properties.

On the basis of tens of thousands of experiments, we postulate the existence of an \emph{integer} $\kappa_3(F)$ which satisfies the following conjecture:

\begin{conj}\label{conj:kappa}
Let $F$ be a number field with signature $(r_1, r_2)$, degree $n = r_1 + 2r_2$, and discriminant $\Delta_F$.  Define $\kappa_3(F)$ to be the real number 
\[\kappa_3(F) = \frac{\sqrt{|\Delta_F|}^{\,5} }{2^{3 r_1+5 r_2}3^{r_1+r_2}\pi^{3 r_2}}\frac{w_3(F) \zeta_F(3)}{\reg^{\bloch}_3(F)}.\]
Then $\kappa_3(F) \in \ZZ_{\geq 1}$.
\end{conj}

With the above notations, we can give a statement in the spirit of
Lichtenbaum's conjecture \eqref{eq:licht} but in the setting of
Zagier's conjecture \eqref{eq:dbz}.

\begin{conj}\label{conj:liza} For a number field $F$ with signature $(r_1, r_2)$ we have
\[\pi^{2(r_1+r_2)} \zeta_F^*(-2) = q_{r_1,r_2}\cdot \frac{\kappa_3(F)}{w_3(F)}{{\reg}^{\bloch}_3(F)}\]
with $q_{r_1,r_2}={(-1)^{r_1}}{2^{r_1+2 r_2} 3^{r_1+r_2}}$.
\end{conj}

We emphasize that although we work with the putative integer $\kappa_3(F)$ and not $\size{K_4(\OO_F)}$, we expect $\kappa_3(F)$ to be closely related---and in fact sometimes equal to---the order of $K_4(\OO_F)$.   In the spirit of Lichtenbaum, we ask 

\begin{quest} \label{qq}
Let $E$ and $F$ be number fields.
\begin{enumerate}
\item \label{q1}
Is it true that if $F\subset E$, then   
\[\kappa_3(F)
\ \text{ divides }\ \kappa_3(E)?\]
    \item \label{q2}
Is it true that 
\[\size{K_4(\OO_F)} \ \text{ divides }\ \kappa_3(F)?\]
\end{enumerate}
\end{quest}
Our data supports the truth of (1).  We know the two  values in \eqref{q2} are not equal in general and are investigating the discrepancy more deeply for a future project.  For partial results and heuristics, see Examples~\ref{ex:d4sqrt5} and \ref{ex:f5}, as well as Section~\ref{s:cyclo}.

\subsection{} Here is a guide to the paper. In Section~\ref{s:basic}
we review constructions from the literature and define the basic
objects we wish to compute, including our version of the weight 3 Bloch group
(Definition~\ref{def:working-bloch-group3}, \eqref{def:our_B_3}) and the important quantities $w_3(F)$
and $\kappa_3(F)$.  In Section \ref{s:id} we discuss implementation
details.  Section~\ref{s:ex} presents various examples of computing $\kappa_3(F)$ for a selection of fields with different signatures.  Section~\ref{s:cyclo}
discusses the expected discrepancy between $\kappa_3(F)$ and
$\size{K_4(\OO_F)}$.  Finally Section \ref{s:zl} discusses using these
results to evaluate the special value $\zeta_F(3)$ to high precision.

\section{The basic objects in the computations}\label{s:basic}

\subsection{}
In this section, we first recall the classical Bloch group as given by Suslin \cite{SuslinK3}.  Then we give a modification by Calegari, Garoufaldis, and Zagier \cite{cgz}.  These are Bloch groups of weight $2$.  Then we give a further modification of the weight $2$ Bloch group that is more accessible to computer calculation, and use it to define a modified higher Bloch group of weight $3$.  

Let $Z(F)$ denote the group of formal $\ZZ$-linear combinations of elements of $F^\times$.  For $x \in F\setminus \set{0,1}$, let $\symb{x}$ denote the image of $x$ in $Z(F)$.  View the multiplicative abelian group $F^\times$ as a $\ZZ$-module, and let $\bwedgetwo F^\times$ denote the wedge product
\[\bwedgetwo F^\times = \frac{F^\times \otimes_\ZZ F^\times}{\gen{x \otimes y + y \otimes x \given x, y \in F^\times}}.\]  For $x$ and $y$ in $F^\times$, we denote the image of $x \otimes y$ in $\bwedgetwo F^\times$ by $x \wedge y$.  Note that this differs from the usual exterior product defined by $x \otimes x = 0$.  This version is denoted $(F^\times \otimes_\ZZ F^\times)_\sigma$ in Suslin's work \cite{SuslinK3}.  We follow the notation of Calegari, Garoufalidis, and Zagier \cite{cgz} and denote it $\bwedgetwo F^\times$.
\begin{df}
The \emph{Bloch group} of a field $F$ is the quotient
\[B_2(F) = A_2(F)/C_2(F),\]
where $A_2(F)$ is the kernel of the map 
\[\delta_2 \colon Z(F) \to \bwedgetwo F^\times, \quad \symb{x} \mapsto 
x \wedge (1-x),\]
and $C_2(F) \subseteq A_2(F)$ is the group generated by the \emph{five-term relation}
\[\xi_{x,y} = \symb{x} -\symb{y} + \symb*{\frac{y}{x}} - \symb*{\frac{1-x^{-1}}{1-y^{-1}}} + \symb*{\frac{1-x}{1-y}},\]
with $x$ and $y$ ranging over $F \setminus \set{0,1}$ with $x \neq y$.
\end{df}

The modification of the Bloch group follows two steps.  First, augment $A_2(F)$ by formally adding symbols $\symb{0}$, $\symb{1}$, and $\symb{\infty}$ with $\symb{1} = 0$ and $\symb{0} = -\symb{\infty} = \symb{x} + \symb{1-x}$ for any $x \in F$.(\footnote{Suslin \cite{SuslinK3} shows $\symb{x} + \symb{1-x}$ is independent of $x$.}) This yields a group that is isomorphic to $B_2(F)$.  Next, we modify the groups  to define a quotient that uses all specializations of the five term relation $\xi_{x,y}$ that do not involve indeterminate terms such as $\frac{0}{0}$ or $\frac{\infty}{\infty}$.  The issue that arises is that the element 
\[\xi_{x,1} = \symb{x}-\symb{1} + \symb*{\frac{1}{x}} - \symb*{\frac{1-x^{-1}}{0}} + \symb*{\frac{1-x}{0}} =  \symb{x} + \symb{x^{-1}}\]
does not necessarily lie in $A_2(F)$, since 
\[\delta_2(\symb{x} + \symb{x^{-1}}) = x \wedge (-x)\]
is potentially a nontrivial $2$-torsion element.  One way to account for this is to modify the exterior product.  

Let 
\[\twedgetwo F^\times = \frac{F^\times \otimes_\ZZ F^\times}{\gen{x \otimes (-x) \given x \in F^\times}}.\]
 For $x$ and $y$ in $F^\times$, we denote the image of $x \otimes y$ in $\twedgetwo F^\times$ by $\wtilde{x}{y}$.  One can easily check that 
\begin{itemize}
\item the elements $\wtilde{x}{x}$ and $\wtilde{x}{(-1)}$ are equal and
are $2$-torsion for all $x \in F^\times$;
\item $\wtilde{y}{x} + \wtilde{x}{y} = 0$ for all $x, y \in F^\times$; and 
\item $\wtilde{\zeta}{\omega}$ is $2$-torsion for any roots of
unity $\zeta, \omega  \in F^\times$.
\end{itemize}

Let $\ZZ[\PP^1(F)]$ denote the group of formal $\ZZ$-linear combinations of elements of $\PP^1(F) = F \cup \set{\infty}$.  As before, for $x \in \PP^1(F)$, let $\symb{x}$ denote the image of $x$ in $\ZZ[\PP^1(F)]$.  

\begin{df}
The \emph{modified Bloch group} of a field $F$ is the quotient
\[\tB_2(F) = \tA_2(F)/\tC_2(F),\]
where $\tA_2(F)$ is the kernel of the map 
\[\td_2 \colon \ZZ[\PP^1(F)] \to \twedgetwo F^\times, \quad \symb{x} \mapsto 
\begin{cases}
0 & \text{if $x \in \set{0, 1, \infty}$},\\
\wtilde{x}{(1-x)} & \text{otherwise},\\
\end{cases}\]
and $\tC_2(F) \subseteq \tA_2(F)$ is the group generated by specializations of the \emph{five-term relation} $\xi_{x,y}$ 
with $x$ and $y$ ranging over $\PP^1(F)$, forbidding only those terms for which the resulting fractions have the indeterminate form $\frac{0}{0}$ or $\frac{\infty}{\infty}$.
\end{df}

\begin{prop}[{\cite[Lemma 2.2]{cgz}}]
There is a surjection $B_2(F) \to \tB_2(F)$ whose kernel is $2$-torsion.
\end{prop}
The extra relations that arise from including the degenerate five term relations imply that in $\tB_2(F)$, we have 
\begin{align*}
\symb{1} &= \symb{x} + \symb{x^{-1}} = 0, \quad \text{for all $x \in \PP^1(F)$},\\
\symb{0} &= -\symb{\infty} = \symb{x} + \symb{1-x}, \quad \text{for any $x \in F^\times$}.\\
\end{align*}
\subsection{The pruned Bloch group}
The Bloch group and its modification are difficult to work with in explicit calculations.  In particular, given an element representing a class in the Bloch group or its modification, there is not an efficient method of checking whether or not the element represents the trivial class.  This is because there is no algorithm known to express a linear combination representing the trivial element $0\in B_2(F)$ as a sum of five term relations. To address this, we define a further modification of $B_2(F)$ that replaces the condition to be trivial in $B_2(F)$ with a numerically testable---which has always been used in practice---condition involving the dilogarithm on elements of $\tA_2(F)$ as explained below.

Let $\Li_m(z)$ be the classical polylogarithm, a multivalued function defined by analytic continuation to $\C\setminus\{0,1\}$, of the series $\sum_{k\geq 1} {z^k}/{k^m}$ ($|z|<1$).
Let $\CL_2 \colon \PP^1(\C) \to \RR$ denote the single-valued Bloch-Wigner dilogarithm
\[\CL_2(z) = \Im(\Li_2(z)) + \log|z| \arg(1-z),\] 
 and $\CL_2(0) = \CL_2(1) = \CL_2(\infty) = 0$.  This is in line with the symmetry $\CL_2(\overline{z}) = -\CL_2(z)$, and hence $\CL_2(x)=0$ for $x\in \RR$.  Each choice of $r_2$ nonconjugate complex embeddings of $F$ induces a $\ZZ$-linear function \[\tCL_2 \colon \ZZ[\PP^1(F)] \to \RR^{r_2}, \quad \symb{x} \mapsto (\CL_2(\sigma(x)))_{\sigma}.\] 
  
The Bloch-Wigner dilogarithm vanishes on $\tC_2(F)$, i.e., \[\tC_2(F) \subset \tA_2(F) \cap \ker(\tCL_2).\]  This condition can be quickly tested numerically to high precision, and so it gives rise to a further modification of the Bloch group that is more amenable to explicit calculations in a computer.  

\begin{df}\label{def:working-bloch-group2}
The \emph{pruned Bloch group} of a number field $F$ is the quotient  
\[\wB_2(F) = \frac{\tA_2(F)}{\tA_2(F) \cap \ker(\tCL_2)},\]
where $\tA_2(F) = \ker(\td_2)$. 
\end{df}

We now give an example showing that the different Bloch groups defined above, while related, are in fact subtly different.
\begin{ex}

Consider the element $\symb{\frac{1}{2}}$.  This can be viewed as an element of $Z(F)$ or $\ZZ[\PP^1(F)]$ for any number field $F$.  

A quick calculation shows $\symb{\frac{1}{2}}$ gives a class in $B_2(F)$ for any number field $F$.  We have 
\[\delta_2(\symb*{\tfrac{1}{2}}) = \tfrac{1}{2} \wedge (1-\tfrac{1}{2}) = \tfrac{1}{2} \wedge \tfrac{1}{2} = 0.\]
On the other hand, $\td_2(\symb*{\tfrac{1}{2}}) = \wtilde{\tfrac{1}{2}}{\tfrac{1}{2}}$ is only $2$-torsion in $\tilde{\wedge}^2 F^\times$, so $\symb{\frac{1}{2}}$ does not generally give a class in $\wB_2(F)$ or $\tB_2(F)$.  However, if $2$ is a square in $F$, then $\wtilde{\tfrac{1}{2}}{\tfrac{1}{2}} = 0$, so we get a class in both $\wB_2(F)$  and $\tB_2(F)$.  In this case, we get the trivial class in $\wB_2(F)$ in view of the vanishing of $\CL_2$ on $\RR$, but it is not necessarily trivial in $\tB_2(F)$.  
\end{ex}

Finally, we turn to defining a Bloch group of weight $3$.  As in Definition~\ref{def:working-bloch-group2}, we aim for a slight modification of Zagier's generalization that is efficiently computable.

Define 
\[\beta \colon \ZZ[\PP^1(F)] \to F^\times \otimes_\ZZ (\twedgetwo F^\times), \quad \symb{x} \mapsto x \otimes \td_2(x).\]  
Each group homomorphism $\nu \colon F^\times \to \ZZ$ induces a $\ZZ$-linear map 
\[\tnu \colon \ZZ[F^\times] \to \ZZ[F^\times], \quad \symb{x} \mapsto \nu(x) \symb{x}.\]  
This extends to a function on $\ZZ[\PP^1(F)]$, which we also call $\tnu$.  Zagier \cite{ZagierTexel} shows if $\xi \in \ker(\beta)$, then $\tnu(\xi) \in \ker(\td_2)$.  In particular, $\tnu(\xi)$ defines a class in $\wB_2(F)$.  Thus we get an induced map 
\begin{equation} \label{eq:delta3}
\td_3 = \prod_\nu \tilde{\nu}\colon \ker(\beta) \to \prod_\nu \wB_2(F),
\end{equation}
where $\nu$ runs over a generating set of $\Hom_\ZZ(F^\times, \ZZ)$.  

Let $\CL_{3} \colon \PP^1(\C) \to \RR$ denote Zagier's single-valued trilogarithm,
\begin{equation}\label{trilog}
\CL_{3} (z) = \Re\big(\Li_3(z) - \log|z| \,\Li_2(z) + \frac13 \log^2|z|
\,\Li_1(z)\big),
\end{equation}
and $\CL_3(0) = \CL_3(\infty) = 0$. 
Each choice of $r_1$ real and $ r_2$ nonconjugate complex embeddings $\sigma$ of $F$ induces a $\ZZ$-linear function 
\[\tCL_3 \colon \ZZ[\PP^1(F)] \to \RR^{r_1 + r_2}, \quad \symb{x} \mapsto (\CL_3(\sigma(x)))_{\sigma}.\]

The trilogarithm satisfies classical $1$-variable functional equations $I_x$, $D_x$, and $T_x$; Gangl's functional equation $A_x$ \cite{gangl-diplom, neilsen-polylogs}; and the $2$-variable Spence-Kummer functional equation $S_{x,y}$ in the sense that 
\[\tCL_3(I_x) = \tCL_3(D_x)= \tCL_3(T_x)=\tCL_3(A_x)= \tCL_3(S_{x,y}) = 0\]
for all $x, y \in F$, $(x,y) \neq (0,0),(1,1)$.  The inversion relation $I_x$ and duplication relation $D_x$ are often called the \emph{trivial relations}.
The functional equations can be expressed as 
\[I_x = \symb{x} - \symb*{\frac{1}{x}}, \quad D_x = \symb{x^2} - 4\symb{x} - 4\symb{-x}, \quad T_x = \symb{x} + \symb{1-x} + \symb*{1 - \frac{1}{x}} - \symb{1},\]
\begin{multline*} A_x=2\bigg([x(1-x)]+\Big[-\frac{x}{(1-x)^2}\Big]+\Big[-\frac{1-x}{x^2}\Big]\bigg) \\ +  3\bigg([1-x(1-x)]+\Big[\frac{1-x(1-x)}{(1-x)^2}\Big]+\Big[\frac{1-x(1-x)}{x^2}\Big] 
\bigg)- 6[1],\end{multline*}
and 
\begin{multline*}S_{x,y} = \symb*{\frac{x(1-y)^2}{y(1-x)^2}} + \symb{xy} + \symb*{\frac{x}{y}} - 2\symb*{\frac{x(1-y)}{y(1-x)}} - 2\symb*{\frac{x(1-y)}{x-1}} \\- 2 \symb*{\frac{y(1-x)}{y-1}} -2 \symb*{\frac{1-y}{1-x}} - 2\symb{x} - 2\symb{y} + 2\symb{1}.\end{multline*}
Most importantly, the trilogarithm satisfies Goncharov's 22-term equation in three variables which, together with its degenerations, conjecturally generates all the $\tCL_3$-functional equations \cite[Section~5]{GoncharovConfigurations}.

\begin{df}\label{def:working-bloch-group3}
The \emph{pruned Bloch group of weight $3$} of a number field $F$ is the quotient 
\begin{equation}\label{def:our_B_3}
\wB_3(F) = \frac{\tA_3(F)}{\tA_3(F) \cap \ker(\tCL_3)},
\end{equation}
where $\tA_3(F) = \ker(\td_3)$.
\end{df}

The diagram below summarizes some of the maps defined so far.  The goal is to compute the top arrow to get $\Lambda_3(F) \subset \RR^{r_1 + r_2}$, the image of $\tA_3(F)$ under the trilogarithm $\tCL_3$.  The \emph{third Bloch group regulator of $F$}, denoted  $\reg_3^{\bloch}(F)$, is defined as the covolume of $\Lambda_3(F)$.
\[
\begin{tikzcd}
 \tA_3(F) \arrow[r, "\tCL_3"] \arrow[d,hook]   & \Lambda_3(F) \\
 \ker(\beta) \arrow[r, hook] \arrow[d, "\td_3"] & \ZZ[\PP^1(F)] \arrow[r, "\beta"] & F^\times \otimes_\ZZ (\twedgetwo F^\times)\\
 \prod \wB_2(F)
\end{tikzcd}
\]

\subsection{}

We now describe the quantity $w_3(F)$ that we use as a stand-in for $\size{K_5(\OO_F)_{\tors}}$.  For $m\geq 1$, 
the number $w_m(F)$ is the size of the cyclic Harris--Segal 
\cite[p.~149]{weibel-handbook5} that describes the bulk of $K_{2m-1}(\OO_F)_{{\rm tors}}$. Formally, one has $w_m(F)$ is the size of the group of \emph{higher roots of unity}, given in terms of Galois cohomology.  We have
\[w_m(F) = \size{H^0(F,\Q/\Z(m))},\]
where the right hand group is the zeroth Galois cohomology group, i.e., the invariants under $\Gal(\overline{F}/F)$) with coefficients in the $m$-th Tate twist applied to the infinite torsion group $\QQ/\ZZ$. 

Up to a power of $2$, $w_m(F)$ is the order of torsion subgroup of $K_{2m-1}(\OO_F)$, and the full group structure of $K_{2m-1}(\OO_F)$ can be described in terms of $w_m(F)$ and the number of real and complex places of $F$.
\begin{thm}[{\cite[Theorem~1, p.~140]{weibel-handbook5}}] \label{thm:Kodd}
Let $F$ be a number field with $r_1$ real places and $r_2$ complex places.  Fix an integer  weight $m \geq 2$.  Let $\OO_{F,S}$ be the ring of $S$-integers in $F$.  Then $K_{2m-1}(\OO_{F,S}) \simeq K_{2m-1}(F)$.  Furthermore,
\begin{enumerate}
\item if $r_1 = 0$, then 
\[K_{2m-1}(\OO_{F,S}) \simeq \ZZ^{r_2} \oplus \ZZ/w_m(F) \ZZ;\]
\item if $r_1 > 0$, then 
\[
K_{2m-1}(\OO_{F,S}) \simeq \begin{cases}
\ZZ^{r_2} \oplus \ZZ/w_m(F)\ZZ, & m \equiv 0 \pmod{4}\\
\ZZ^{r_1 + r_2} \oplus \ZZ/w_m(F)\ZZ, & m \equiv 1 \pmod{4}\\
\ZZ^{r_2} \oplus \ZZ/2w_m(F)\ZZ \oplus (\ZZ/2\ZZ)^{r_1 - 1}, & m \equiv 2 \pmod{4}\\
\ZZ^{r_1 + r_2} \oplus \ZZ/\tfrac{1}{2}w_m(F)\ZZ, & m \equiv 3 \pmod{4}.
\end{cases}
\]
\end{enumerate}
\end{thm}

Our main interest is weight $m = 3$ and the full ring of integers $\OO_F$.
\begin{cor}\label{w3F}
Let $F$ be a number field with $r_1$ real places and ring of integers $\OO_F$.  Then 
\[w_3(F) = \begin{cases}
\size{K_{5}(\OO_F)_{\tor}}, & \text{if $r_1 = 0$}\\
2\size{K_{5}(\OO_F)_{\tor}} = 2, & \text{if $r_1 > 0$.}
\end{cases}
\]
\end{cor}
For an indication on how we can compute $w_3(F)$, see Section~\ref{computing_wmF} below.

\subsection{The crucial quantity $\kappa_3(F)$}

Let $F$ be a number field with signature $(r_1, r_2)$, degree $n = r_1 + 2r_2$, and discriminant $\Delta_F$. 
We expect the following quantity $\kappa_3(F)$, motivated by the conjectures of Lichtenbaum and Zagier, to be closely related to the order of $K_4(\CO_F)$. More explicitly, as in the introduction we define   $\kappa_3(F)$ to be the positive real number 
\begin{equation}\label{eq:kappa3F}
    \kappa_3(F) = \frac{\sqrt{|\Delta_F|}^{\,5} }{2^{3 r_1+5 r_2}3^{r_1+r_2}\pi^{3 r_2}}\frac{w_3(F) \zeta_F(3)}{\reg^{\bloch}_3(F)}.
\end{equation}

Using the functional equation of $\zeta_F$ and writing its first non-vanishing coefficient of the Taylor series  at $s=-2$, where it vanishes to order $r_1+r_2$, in terms of its derivative $\zeta_F^{(r_1+r_2)}$ as
\[\zeta_F^*(-2) = \frac{\zeta_F^{(r_1+r_2)}(-2)}{(r_1+r_2)!},\] we get the---slightly easier to memorize---form
\[\kappa_3(F) =(-1)^{r_1} \frac{ \zeta(2)^{r_1+r_2}}{2^{r_2}}\frac{w_3(F)\zeta_F^*(-2)}{{\reg}^{\bloch}_3(F)}\,.\]

In particular we note that if $\zeta^*_F(1-m)= \zeta_F(1-m)\not=0$, the quantity $w_{m}(F)\zeta^*_F(1-m)$ is expected to be an integer for $m\geq 0$.  Indeed, the special case $m=2$ is the famous Birch--Tate Conjecture, which has been proved for totally real abelian fields.  However if $\zeta_F(1-m)=0$, then one must divide this quantity by the (presumably transcendental) regulator $\reg^{\bloch}_m(F)$, as well as additional powers of $\pi$, to obtain a rational number.

For the convenience of the reader we reiterate the conjecture stated in the introduction.  We urge the impatient reader to turn to Section \ref{s:ex} to see explicit examples of computing $\kappa_3(F)$.

\begin{conj}\label{conj:kappa_revisit}
For a number field $F$ we have  $$\kappa_3(F) \in \Z_{\geq 1}.$$
\end{conj}

\section{Implementation details}\label{s:id}
\subsection{}
While the groups $\wB_2 (F)$ and $\wB_3 (F)$ defined above are more amenable to computation than their Bloch group counterparts,  they are still defined by infinitely many generators and relations.  In this section, we describe the techniques that we use to approximate them with finite calculation.  

\subsection{}
Recall that our  goal is to compute  $\Lambda_3(F) \subset \RR^{r_1 + r_2}$, the image of $\tA_3(F)$ under the trilogarithm $\tCL_3$.  The general computational strategy is to construct a finite rank submodule $N \subset \ZZ[\PP^1(F^\times)]$ such that the intersection $\ker(\beta\vert_{N}) \cap \tA_3(F)$ is large, in the sense that we are computationally convinced that the trilogarithm image is the full lattice $\Lambda_3(F)$.

The actual choice of $N$ is a delicate art, which we address later.  For now, fix a finite rank submodule $N \subset \ZZ[\PP^1(F^\times)]$.  The image $\beta(N)$ is a finite rank submodule of $F^\times \otimes_\ZZ (\twedgetwo F^\times)$.  Choose $N$ so that there is a large nontrivial submodule $M \subset N$ in the kernel of $\beta$. 
Here, large is in the sense that the image under $\tCL_3$ of the intersection of $M$ with $\tA_3(F)$ has rank $r_1 + r_2$.  Note that $M$ can be computed explicitly from $N$ and $\beta$.   Find linear relations among the  $\beta(\xi)$ as $\xi$ ranges over a generating set for $N$ to construct $M = \ker(\beta\vert_{N})$.  This is an algebraic calculation that can be done with exact arithmetic in a computer.

\subsection{}
Next, we restrict $\td_3$ to $M$, and turn our attention to computing the submodule of $M$ that is in the kernel of $\td_3$.  In other words, we want to compute $M_0 = M \cap \tA_3(F)$.  Note that since $M$ has finite rank, there are only finitely many $\nu \in \Hom_\ZZ(F^\times, \ZZ)$ that need to be considered in \eqref{eq:delta3}.  Specifically, let $W \subset F^\times$ be a finite rank submodule such that $\beta(M)$ is contained in $W \otimes_\ZZ (\twedgetwo F^\times)$.  Then, in the product map that defines $\td_3$, it suffices to let $\nu$ range over a finite generating set $H \subset \Hom_\ZZ(W,\ZZ)$.  In order for $\xi \in M$ to be in $\tA_3(F)$, we need that for each $\nu \in H$, $\tCL_2(\tnu(\xi)) = 0 \in \RR^{r_2}$. To this end, we let $\set{\xi_i \given i \in I}$ be a generating set for $M$.  For $\nu \in H$, use linear relations among $\set{\tCL_2(\tnu(\xi_i)) \given i \in I} \subset \RR^{r_2}$ to compute $M_0$.  This calculation requires linear algebra using approximations of trilogarithm values on a computer, so there are numerical linear algebra issues that must be taken into account.  Then $\tCL_3(M_0)$ generates a sublattice of  $\Lambda_3(F) \subset \RR^{r_1 + r_2}$.  

\subsection{}
One can use computations on subfields of $F$ to speed up calculations for $F$.  Let $X_F \subset M_0$ denote a set such that $\tCL(X_F)$ is a basis for $\Lambda_3(F)$.  Note that $X_F$ is a subset of $\ZZ[F^\times]$ of cardinality $r_1 + r_2$.  In particular, we can store $X_F$ exactly in a computer, which allows us to later calculate $\tCL_3(\wB_3(F))$ to the desired precision.  The data of $X_F$ can be used in the calculation for any field extension of $F$.  For example, suppose $F_0$ is a subfield of $F$.  Then viewing $\ZZ[\PP^1(F_0)]$ as a submodule of $\ZZ[\PP^1(F)]$, we see that for each $\xi \in X_{F_0}$, the trilogarithm image $\tCL_3(\xi)$ is a nonzero vector in $\Lambda_3(F)$.
In particular, in any choice of $N$, it is beneficial to include the elements of $X_{F_0}$ for all subfields $F_0 \subset F$.

\subsection{}
As noted above, the choice of $N$ is delicate.  For example, we see that we want to choose $N$ so that the rank of $N$ is much larger than the rank of the image $\beta_3(N)$ so that the kernel $\ker(\beta\vert_{N})$ is large,  while at the same time arranging that this kernel has at least a rank $r_1 + r_2$ intersection with $\ker(\td_3)$.  We describe a method of choosing $N$ below using exceptional $S$-units.

\begin{df}  Let $S$ be a finite set of places of $F$, including all of the archimedean places.  The ring $\OO_S$ of \emph{$S$-integers} in $F$ is the set of $x \in F$ such that $x$ is $v$-integral for all $v \not \in S$; that is 
\[\OO_S = \bigl\{ x \in F \bigm|  \text{$\norm{x}_v \leq 1$ for all $v\notin S$}\bigr\}.\]  An element $x \in F^\times$ is an \emph{$S$-unit} if both $x$ and $1/x$ are $S$-integers.  An element $x \in F^\times \setminus \set{1}$ is an \emph{exceptional $S$-unit} if both $x$ and $1-x$ are $S$-units.
\end{df}

Let $S$ be a finite set of places of $F$, including all of the archimedean places.  Let $\XX_S$ denote the set of exceptional $S$-units in $F$.  Elements of $\XX_S$ can be readily computed using software such as Magma \cite{magma} or gp/PARI \cite{PARI2}.  For example, let $\set{\zeta, x_1, x_2, \dots, x_N}$ be an $S$-unit basis, where $\zeta$ generates the roots of unity in $F^\times$.  Consider elements of the form 
\[x = \zeta^{e_0} x_1^{e_1} x_2^{e_2} \cdots x_d^{e_d},\]
for integers $e_i$.  By choosing different $e_i$ and including more places in $S$ as needed, one can quickly find many exceptional $S$-units.

\subsection{}
We can finally explain the computation of $N$.  We 
choose a finite subset $X \subseteq \XX_S$, and take $N \subset \ZZ[\PP^1(F^\times)]$ to be the submodule generated by $\set{[x] \given x \in X}$ together with $X_{F_0}$ for every subfield $F_0 \subset F$.  In practice, we take $\size{X}$ to be between $500$ and $1000$,  and sometimes increase up to $1200$ for higher degree fields.  However, there is no easy way to ensure success; in most cases experience and luck play a substantial role.

\begin{rem} We list some tips and comments below, based on the first-named author's experience computing thousands of examples:
\begin{itemize}
\item The computational hurdle in the initial linear algebra calculation stems from the size of the calculation.  If we omit the subfield information, one must find integral relations among the $\size{X}$ elements in $F^\times \otimes_\ZZ (\twedgetwo F^\times)$.  The natural basis to choose for the submodule of $F^\times \otimes_\ZZ (\twedgetwo F^\times)$ comes from the $S$-unit basis, so the linear relation calculation is occuring in a space of rank roughly $d^3$, where $d$ is the number of elements in the $S$-unit basis.
\item It is beneficial to fix a set of rational primes $S_\QQ$ and take $S$ to be the set of all the places that lie above the primes in $S_\QQ$.  This speeds up the check to see if $1-x$ is an $S$-unit, as the check can be done with valuations from the primes in $S_\QQ$ on the norm of $1-x$.  This has the downside of increasing $\size{S}$, which increases the computational effort required in the  linear algebra step.
\item We find more success when the primes $2$ and $3$ are included in $S_\QQ$.  
\item We find more success, especially when the degree of $F$ is large, when $X$ contains many elements $[x]$ with $x$ supported on fewer primes.  In other words, $x$ is an $S'$-unit and $1-x$ is an $S$-unit, where $S'$ is a proper subset of $S$.
\end{itemize}
\end{rem}
\subsection{Computing $w_m(F)$}\label{computing_wmF}
In elementary terms $w_m(F)$ is the largest integer $w$ such that for every $\sigma\in \Gal(\overline{F}/F)$ we have 
\[\sigma(\zeta_n)= \zeta_n^{\chi(\sigma)^m}\]
with the cyclotomic character $\chi \colon \Gal(\overline{F}/F) \to (\Z/w\Z)^\times$.
So for any unramified prime ideal $\mathcal P$ we get that $\chi(F_{\mathcal P}) \equiv N({\mathcal P}) \pmod w$ and the invariance under $F_{\mathcal P}$ translates into the condition $w\mid N(\mathcal P)^m-1$, whereas for ramified primes $p$ we have no constraint, and one writes $N(\mathcal P)^\infty$.
The above leads to the formula 
\[w_m(F) = \gcd_{\mathcal P} N(\mathcal P)^\infty \big(N(\mathcal P)^m-1\big)\,,\]
where $\mathcal P$ runs through all the prime ideals in $\CO_F$.  In practice, we take $\mathcal P$ to run over all primes with norm $\leq 2000$.

Voight \cite{voight-relative-class-numbers} gives an alternative characterization of $w_m(F)$ for odd weight $m$.  For a prime power $q$, let $\QQ(\zeta_q)^{(m)}$ denote unique the subfield of $\QQ(\zeta_q)$ of index $\gcd(\phi(q),m)$, that is we have $\gcd(\phi(q),m) = [\QQ(\zeta_q) : \QQ(\zeta_q)^{(m)}]$.

\begin{lem}[{\cite[Lemma~1.1]{voight-relative-class-numbers}}] \label{lem:wmF}
Let $F$ be a number field.  Let $q$ be a prime power, and let $m \geq 3$ be an odd integer.  Then $q \mid w_m(F)$ if and only if $F$ contains $\QQ(\zeta_q)^{(m)}$.
\end{lem}

We believe the following propositions are well-known.  We include the proofs for the convenience of the reader.

\begin{prop} \label{prop:wmF2}
Let $F$ be a number field with at least one real place, and let $m \geq 3$ be an odd integer.  Then $w_m(F) = 2$.
\end{prop}

\begin{proof}
We use Lemma~\ref{lem:wmF}.  First, we have $2 \mid w_m(F)$.  Since $\QQ(\zeta_2) = \QQ$ and $\gcd(\phi(2),m) = 1$, we have $\QQ(\zeta_2)^{(m)} = \QQ \subseteq F$.  Next, we show $2^2 \nmid w_m(F)$.  We have $\gcd(\phi(2^2),m) = 1$ since $m$ is odd. Then $\QQ(\zeta_{2^2})^{(m)} = \QQ(i)$ cannot be contained in $F$, since $F$ has a real place.  Finally, suppose $q$ is an odd prime power.  Then $\gcd(\phi(2^2),m) = [\QQ(\zeta_q) : \QQ(\zeta_q)^{(m)}]$ is odd, so $\Gal(\QQ(\zeta_q) /\QQ(\zeta_q)^{(m)})$ cannot contain complex conjugation.  Thus $\QQ(\zeta_q)^{(m)}$ is totally complex and cannot be contained in $F$.
\end{proof}

\begin{prop} \label{prop:w3cyclo}
Let $n$ be a positive integer.
\[w_3(\QQ(\zeta_n)) = \begin{cases} 
\lcm(2,n) & \text{if $3 \nmid n$,}\\
3\lcm(2,n) & \text{if $3 \mid n$.}
\end{cases}\]
\end{prop}
\begin{proof}
We have that $w_3(\QQ(\zeta_n))$ is the largest integer $q$ such that
the Galois group $\Gal(\QQ(\zeta_n,\zeta_q)/\QQ(\zeta_n))$ has
exponent dividing $3$.  It follows that $\lcm(2,n) \mid
w_3(\QQ(\zeta_n))$.  Suppose $3 \nmid n$.  Then
$\Gal(\QQ(\zeta_n,\zeta_3)/\QQ(\zeta_n))$ has order $2$, so $3 \nmid
w_3(\QQ(\zeta_n))$.  Now suppose $3^k \mathrel\Vert n$ for some $k >
0$.  Then $\Gal(\QQ(\zeta_n,\zeta_{3^{k+j}})/\QQ(\zeta_n))$ has
exponent $3^j$, so $j = 1$ gives $3^{k + 1}$ is the largest power of
$3$ that divides $w_3(F)$.  It follows that $w_3(\QQ(\zeta_n)) =
3\lcm(2,n)$ in this case.
\end{proof}
\section{Examples}\label{s:ex}

\subsection{} In this section we give some examples to illustrate
various features of our computational data.  We consider a range of
degrees and signatures; the largest example (Example \ref{ex:d14}) is
a Galois extension with group $D_{14}$.  

Altogether we computed $\kappa_3(F)$ for $43847$ fields of degrees
ranging from $2$ to $14$ and with all possible signature types and for
a variety of Galois closures, as well as some other fields of larger
degrees.  Table \ref{tab:range} shows how many fields of the different
degrees we treated.  Naturally for larger degrees the computations
became more difficult and we must settle for fewer examples.  We also
had to settle for imbalance in the signature types.

\begin{table}[htb]
\caption{The range of our computations.\label{tab:range}}
\centering
\begin{tabular}{|c|r||c|r||c|r|}
\hline
$n$ & Number of fields & $n$ & Number of fields & $n$ & Number of fields \\
\hline
2 & 21325 & 7 & 1554 & 12 & 505 \\
3 & 1901 & 8 & 1124 & 13 & 64 \\
4 & 2346 & 9 & 318 & 14 & 351 \\
5 & 8607 & 10 & 600 & & \\
6 & 5045 & 11 & 107 & & \\
\hline
\end{tabular}
\end{table}

\begin{ex}\label{ex:QQ}
We begin with an example over $\QQ$.  Consider $\xi = \symb{1} +2 \symb{-1} - 2 \symb{2} + 2 \symb{3} - \symb{-3} \in \ZZ[\PP^1(\QQ)]$.  Then $\td_3(\xi) = 0$, so $\xi$ gives a class in $\wB_3(\QQ)$.  We compute 
\[\tCL_3(\xi) = -0.100171408596632857116644846792620832 \dots.\] 
This value appears to equal $-{\zeta(3)}/{12}$.  To prove this it suffices, because $\tCL_3(\symb{1}) = \zeta(3)$,  to express $12 \xi + \symb{1}$ as a linear combination of functional equations.  One can compute 
\[12 \xi + \symb{1} = 4(A_{2}-T_{2}-3T_{4})\]
up to trivial functional equations: two duplication relations $4D_{2}-7D_{1}$ and two  inversion relations $8I_{4}-4I_{2}$.

Since $r_1 + r_2 = 1$, the point $\tCL_3(\xi)$ generates a finite index sublattice of $\Lambda_3(\QQ)$.  In fact, we believe the index is $1$ because extensive calculations could not improve the lattice. 
Thus, with $\kappa_3$ as in Conjecture~\ref{conj:kappa_revisit}, we have
\[\reg_3^{\bloch}(\QQ) \ceq  \frac{\zeta(3)}{12} \quad \text{and} \quad  \kappa_3(\QQ) \ceq 1.\]

Compare this result with \cite[p.~410]{ZagierTexel}, where Zagier concludes that since the trilogarithm of 
 \[\eta =4\symb{\tfrac{1}{3}}-12\symb{\tfrac{1}{6}} - 3 \symb{\tfrac{1}{8}} + 3 \symb{\tfrac{1}{36}} + 2 \symb{-\tfrac{1}{2}} - 18\symb{-\tfrac{1}{6}} + \symb{-\tfrac{1}{48}} + \symb{\tfrac{2}{9}}\] 
is (numerically) ${89}\zeta(3)/24$, the lattice he finds is the $\ZZ$-span of ${\zeta(3)}/{24}$.  Our lattice has index $2$ in Zagier's, but this does not contradict our result.  Indeed, $\beta(\eta)$ is a nontrivial $2$-torsion element in $\QQ^\times \otimes_\ZZ (\twedgetwo \QQ^\times)$, so does not give a class in our $\wB_3(\QQ)$.  The difference is that Zagier essentially tensors the codomain of $\beta$ with $\QQ$ when defining his Bloch group, thereby forfeiting any claim on torsion properties.
\end{ex}

\begin{rem}\label{rem:uppbound}
The element $\xi$ from Example~\ref{ex:QQ} gives rise to a nontrivial element of $\wB_3(F)$ for any number field $F$.  In particular, for imaginary quadratic fields, this immediately gives an upper bound on the higher Bloch group regulator $\reg_3^{\bloch}(F)$.  In many instances, we believe the bound is tight.  Explicit calculation in small cases shows that there are instances where we can do better (cf. Example~\ref{ex:neg1}).
\end{rem}

\begin{ex} \label{ex:neg1}
We now consider the imaginary quadratic field $\QQ(\sqrt{-1})$.  The element $\xi\in \ZZ[\PP^1(\QQ(\sqrt{-1}))] $ given by
$$\xi = -4\symb{\sqrt{-1} + 1}+4\symb{3}-2\symb{-3}-4\symb{-\sqrt{-1} + 1} $$ 
is in the kernel of $\beta$, and has trilogarithm image \[\tCL_3(\xi) = -0.0500857042983164285583224233963104\dots.\]
This value apparently equals $-{\zeta(3)}/{24}$.  As in Example~\ref{ex:QQ}, we use functional equations to get 
\[24\xi + \symb{1} = 2(5A_{\sqrt{-1}}+8A_{2})-4(29T_{\sqrt{-1}}+12T_{4})\]
up to trivial functional equations: two duplication relations $16D_{-2}-7D_{\sqrt{-1}}$ and three inversion relations $32I_{4}+57I_{2}+58I_{\sqrt{-1}}$.
Thus, \[\reg_3^{\bloch}(\QQ(\sqrt{-1})) \ceq \frac{\zeta(3)}{24} \quad \text{and} \quad \kappa_3(\QQ(\sqrt{-1})) \ceq 1.\]
Note that this improves the upper bound of Remark \ref{rem:uppbound} by a factor of $2$.
\end{ex}

\begin{ex}\label{ex:23}
Now we turn to a mixed signature cubic field, namely the number field \href{https://www.lmfdb.org/NumberField/3.1.23.1}{$F=\QQ(\alpha)$}, where $\alpha$ is a root of $x^3 - x^2 + 1$.  The cubic field  $\QQ(\alpha)$ has signature $(1,1)$ and discriminant $-23$.  

Consider the following elements of $\ZZ[\PP^1(\QQ(\alpha))]$:
\[
\xi_1 = 10 \symb{-1} + 6\symb{2} - 1 \symb{-3} + 2 \symb{3} 
\]
and 
\[
\xi_2 = 2\symb{-1}+\symb{\alpha^{-1}}-2\symb{\alpha^{-2}}+2\symb{-\alpha^{-3}}-2\symb{-\alpha^{-4}}+2\symb{-\alpha^{-5}}.\]

Note that $\xi_1$ is the element from $\ZZ[\PP^1(\QQ)]$ in Example~\ref{ex:QQ}, lifted to $\ZZ[\PP^1(\QQ(\alpha))]$ (the only proper subfield of $F$ is $\QQ$).  We compute 
\begin{align*}
\tCL_3(\xi_1) &= (-0.100171408596632857 \dots,-0.100171408596632857\dots)\\
\tCL_3(\xi_2) &= (
0.279197697898645190 \dots,
-0.440113074739221166 \dots)
\end{align*}
and 
\[\reg_3^{\bloch}(\QQ(\alpha)) \ceq 0.07205437331386738847588301163723266603099882\dots,\]
which gives \[\kappa_3(\QQ(\alpha)) \ceq 1.\]

This example was also treated by Zagier \cite[p.~403]{ZagierTexel}.  He gives four elements of $\ZZ[\PP^1(F)]$ that are linear combinations of $T = \set{\theta, -\theta, \theta^2, \theta^3, -\theta^4, \theta^5}$, where $\theta = -\alpha^{-1}$ is a root of $x^3 - x - 1$.  The first three are in the kernel of $\beta$, so give rise to elements in our $\wB_3(F)$.  The fourth has nontrivial $2$-torsion image in $F^\times \otimes_\ZZ (\twedgetwo F^\times)$, so does not give an element in $\wB_3(F)$.  Using the submodule generated by $\symb{x}$ for $x \in T$, we get an index $12$ sublattice of $\Lambda_3(F)$.
\end{ex}

\begin{ex}\label{ex:neg23closure}
The number $\kappa_3$ is not always $1$.  
Indeed, 
consider the number field \href{https://www.lmfdb.org/NumberField/6.0.12167.1}{$\QQ(\beta)$}, where $\beta$ is a root of $x^6 - 3x^5 + 5x^4 - 5x^3 + 5x^2 - 3x + 1$.  This is a totally complex $S_3$-Galois extension that arises as the Galois closure of the cubic field $\QQ(\alpha)$ of discriminant $-23$ in Example~\ref{ex:23}.  The field $\QQ(\beta)$ has three cubic subfields isomorphic to $\QQ(\alpha)$, coming from the three roots of $x^3 - x^2 + 1$.  It also contains the imaginary quadratic field $\QQ(\sqrt{-23})$ (in fact, $\QQ(\beta)$ is the Hilbert class field of this imaginary quadratic field).  
One can show, through computations similar to those in Example \ref{ex:neg1}, that \[\reg_3^{\bloch}(\QQ(\sqrt{-23})) \ceq \frac{\zeta(3)}{12} \quad \text{and} \quad \kappa_3(\QQ(\sqrt{-23})) \ceq 24. \]

Now consider the elements $\xi_1$ and $\xi_2$ of Example~\ref{ex:23}, lifted to $\ZZ[\PP^1(\QQ(\beta)]$.  To ease the notation, define $r, s, t \in \RR$ by 
\[
\tCL_3(\xi_1) = (r, r, r) \quad \text{and} \quad \tCL_3(\xi_2) = (s, t, t).
\]
The other incarnations of $\xi_2$ in $\ZZ[\PP^1(\QQ(\beta))]$ have trilogarithm images that are cyclic permutations of $\tCL_3(\xi_2)$.  We compute that the Bloch group regulator is 
\[\reg_3^{\bloch}(\QQ(\beta)) 
\ceq 0.05182948694033521008275197754585624579288698088 \dots\]
and \[\kappa_3(\QQ(\beta)) \ceq 24.\]

Note that since $(s, t, t)$ and its cyclic permutations lie in $\Lambda_3(\QQ(\beta))$ and $(r, r, r)$ also lies in $\Lambda_3(\QQ(\beta))$, we must have that $s + 2t$ is rationally related to $r$.  Indeed, one can verify that apparently $s + 2t = 3r$.

Compare the cubic subfield $\QQ(\alpha)$ from Example~\ref{ex:23} with $\QQ(\beta)$,
\[\reg_3^{\bloch}(\QQ(\alpha)) \ceq \abs{r(t-s)} \quad \text{and} \quad \reg_3^{\bloch}(\QQ(\beta)) \ceq \abs{r(t-s)^2}.\]

Given that $\QQ(\beta)$ is totally complex and both $2$ and $3$ are unramified and have two prime ideals above them, Poitou-Tate duality as used in \cite[page~41]{kolster-movahhedi} implies that $24$ divides the order of $K_4$.  Further calculations suggest that the order of $K_4$ is exactly $24$, which is compatible with our expectation in Question~\ref{qq}.
\end{ex}

\begin{ex}
Consider the triquadratic field
$\href{https://www.lmfdb.org/NumberField/8.0.5308416.1}{8.0.5308416.1}
= \QQ(\zeta_{24})$.  We give the result of the calculations for it and
its subfields in Table \ref{tab:qzeta24}.  In this and the following
tables, we provide LMFDB names for fields and links to their entries
(if available).

\begin{table}[htb]
\caption{The triquadratic field $\QQ(\zeta_{24})$.\label{tab:qzeta24}}
\centering
\begin{tabular}{c r r r}
\toprule
$F$ & $\reg_3^{\bloch}(F)$ & $w_3(F)$ & $\kappa_3(F)$ \\ \midrule 
\href{https://www.lmfdb.org/NumberField/1.1.1.1}{$\QQ$} & 0.1001714086 & 2 & $1$ \\ 
\href{https://www.lmfdb.org/NumberField/2.0.3.1}{$\QQ(\sqrt{-3})$} & 0.1001714086 & 18 & $1$ \\ 
\href{https://www.lmfdb.org/NumberField/2.0.4.1}{$\QQ(i)$} & 0.0500857043 & 4 & $1$ \\ 
\href{https://www.lmfdb.org/NumberField/2.0.8.1}{$\QQ(\sqrt{-2})$} & 0.0500857043 & 2 & $3$ \\ 
\href{https://www.lmfdb.org/NumberField/2.0.24.1}{$\QQ(\sqrt{-6})$} & 0.1001714086 & 2 & $23$ \\ 
\href{https://www.lmfdb.org/NumberField/2.2.8.1}{$\QQ(\sqrt{2})$} & 0.1810237108 & 2 & $2^2$ \\ 
\href{https://www.lmfdb.org/NumberField/2.2.12.1}{$\QQ(\sqrt{3})$} & 2.0612867150 & 2 & $1$ \\ 
\href{https://www.lmfdb.org/NumberField/2.2.24.1}{$\QQ(\sqrt{6})$} & 11.8241274897 & 2 & $1$ \\ 
\href{https://www.lmfdb.org/NumberField/4.0.144.1}{$\QQ(\zeta_{12})$} & 0.2576608394 & 36 & $2^2$ \\ 
\href{https://www.lmfdb.org/NumberField/4.0.256.1}{$\QQ(\zeta_{8})$} & 0.0226279639 & 8 & $2^4 \cdot 3$ \\ 
\href{https://www.lmfdb.org/NumberField/4.0.576.1}{$\QQ(\sqrt{-2},\sqrt{-3})$} & 2.9560318724 & 18 & $2 \cdot 3$ \\ 
\href{https://www.lmfdb.org/NumberField/4.0.576.2}{$\QQ(\sqrt{2},\sqrt{-3})$} & 0.1810237108 & 18 & $2^2 \cdot 23$ \\ 
\href{https://www.lmfdb.org/NumberField/4.0.2304.1}{$\QQ(\sqrt{-2},\sqrt{3})$} & 1.0306433575 & 2 & $3 \cdot 23$ \\ 
\href{https://www.lmfdb.org/NumberField/4.0.2304.2}{$\QQ(i,\sqrt{6})$} & 1.4780159362 & 4 & $2^2 \cdot 23$ \\ 
\href{https://www.lmfdb.org/NumberField/4.4.2304.1}{$\QQ(\zeta_{24})^+ = \QQ(\sqrt{2}, \sqrt{3})$} & 109.9247327127 & 2 & $2^4$ \\ 
\href{https://www.lmfdb.org/NumberField/8.0.5308416.1}{$\QQ(\zeta_{24})$} & 3.4351478973 & 72 & $2^8 \cdot 3 \cdot 23$ \\ 
\bottomrule
\end{tabular}
\end{table}
\end{ex}

\begin{ex}\label{ex:d4sqrt5}
Consider the $D_4$-Galois extension $F_{8} = \QQ [x]/ (x^8 - x^7 + 8x^6 + 7x^5 + 19x^4 + 7x^3 + 8x^2 - x + 1)$, given in
LMFDB as 
\href{https://www.lmfdb.org/NumberField/8.0.81450625.1}{8.0.81450625.1}.
We give the result of the calculations for it and its subfields in Table \ref{tab:d4.5s}.

\begin{table}[htb]
    \caption{The $D_4$-extension \href{https://www.lmfdb.org/NumberField/8.0.81450625.1}{8.0.81450625.1}.}
    \label{tab:d4.5s}
    \centering
\begin{tabular}{c r r r}
\toprule
$F$ & $\reg_3^{\bloch}(F)$ & $w_3(F)$ & $\kappa_3(F)$ \\
\midrule
\href{https://www.lmfdb.org/NumberField/1.1.1.1}{$\QQ$} & 0.1001714086 & 2 & $1$ \\ 
\href{https://www.lmfdb.org/NumberField/2.0.19.1}{$\QQ(\sqrt{-19})$} & 0.1001714086 & 2 & 11 \\ 
\href{https://www.lmfdb.org/NumberField/2.0.95.1}{$\QQ(\sqrt{-95})$} & 0.1001714086 & 2 & $2^3 \cdot 3 \cdot 5 \cdot 7$ \\ 
\href{https://www.lmfdb.org/NumberField/2.2.5.1}{$\QQ(\sqrt{5})$} & 0.0079780111 & 2 & $5^2$ \\ 
\href{https://www.lmfdb.org/NumberField/4.0.1805.1}{$\QQ(\sqrt{-2+2\sqrt{-19}})$} & 0.5860457330 & 2 & $5 \cdot 11$ \\ 
\href{https://www.lmfdb.org/NumberField/4.0.9025.1}{$\QQ(\sqrt{5},\sqrt{-19})$} & 0.0079780111 & 2 & $2^3 \cdot 3 \cdot 5^3 \cdot 7 \cdot 11$ \\ 
\href{https://www.lmfdb.org/NumberField/4.2.475.1}{$\QQ(\sqrt{-1+2\sqrt{5}})$} & 0.0093349578 & 2 & $5^4$ \\ 
\href{https://www.lmfdb.org/NumberField/8.0.81450625.1}{$F_8$} & 0.0109227020 & 2 & $2^3 \cdot 3 \cdot 5^7 \cdot 7 \cdot 11$\\ 
\bottomrule
\end{tabular}
\end{table}

We note that the regulators here are surprisingly small, and thus the
computed $\kappa_3$ are very large.  We expect this is a consequence
of the non-trivial $5$-torsion in the Bloch group for $\QQ(\sqrt{5})$,
which propagates to overfields since we have killed lots of torsion when building the pruned Bloch group.
Thus one might expect there to be a discrepancy between these values
and the true order of $K_4$.  On the other hand, we note for instance
that both 7 and 11 divide $144375000 = 2^3\cdot 3 \cdot 5^7 \cdot 7
\cdot 11$, so we expect these primes to appear in
$\size{K_4(\OO_{F_8})}$.  We also note that the field $\QQ(\sqrt{5})$
is the totally real subfield of the cyclotomic field $\QQ(\zeta_5)$;
see Section~\ref{s:cyclo} for more discussion about the influence of
cyclotomic fields on the discrepancy between $\kappa_3$ and the order
of $K_4$.
\end{ex}

\begin{ex}
Consider the $D_4$-Galois extension $F_{8} = \QQ [x]/ (x^8 + 13x^6 +
189x^4 + 208x^2 + 256)$, given in LMFDB as
\href{https://www.lmfdb.org/NumberField/8.0.193220905761.2}{8.0.193220905761.2}.
For this field $F_{8}$ and its subfields, we have $w_3 = 2$.  Thus we have
omitted that column in Table \ref{tab:d4.6481} to save space.  Note that the integer $\kappa_3(F_8)$
factors as $41685792 = 2^5\cdot 3\cdot 67 \cdot 6481$.

\begin{table}[htb]
\caption{The $D_{4}$-extension $F_8$ with LMFDB entry \href{https://www.lmfdb.org/NumberField/8.0.193220905761.2}{8.0.193220905761.2}.}
\label{tab:d4.6481}
\centering
\begin{tabular}{c r r}
\toprule
$F$ & $\text{reg}_3^{\bloch}(F)$ & $\kappa_3(F)$ \\
\midrule
 \href{https://www.lmfdb.org/NumberField/1.1.1.1}{$\QQ$} & 0.1001714086 & $1$ \\ 
 \href{https://www.lmfdb.org/NumberField/2.0.51.1}{$\QQ(\sqrt{-51})$} & 0.1001714086 &  $2 \cdot 67$ \\ 
 \href{https://www.lmfdb.org/NumberField/2.0.663.1}{$\QQ(\sqrt{-663})$} & 0.1001714086 &  $2^4 \cdot 6481$ \\ 
 \href{https://www.lmfdb.org/NumberField/2.2.13.1}{$\QQ(\sqrt{13})$} & 0.7735819184 &  $3$ \\ 
 \href{https://www.lmfdb.org/NumberField/4.0.33813.3}{$\QQ(\sqrt{11+2\sqrt{-51}})$} & 361.8336613633 &  $2 \cdot 67$ \\ 
\href{https://www.lmfdb.org/NumberField/4.0.439569.3}{ $\QQ(\sqrt{13},\sqrt{-51})$} & 0.7735819184 &  $2^5 \cdot 3 \cdot 67 \cdot 6481$ \\ 
 \href{https://www.lmfdb.org/NumberField/4.2.8619.1}{$\QQ(\sqrt{22 + 10\sqrt{13}})$} & 2794.2901254002 &  $3$ \\ 
\href{https://www.lmfdb.org/NumberField/8.0.193220905761.2}{$F_{8}$} & 10093381.3465288987 & $2^5 \cdot 3 \cdot 67 \cdot 6481$ \\ 
\bottomrule
\end{tabular}
\end{table}
\end{ex}

\begin{ex}\label{ex:f5}
Let $F_{5}$ be the Frobenius group of order $20$, given by the
semidirect product $C_{5}\rtimes C_{4}$.  In Table \ref{tab:frob} we
give data for the $F_5$-Galois extension $F_{20}$ defined by the polynomial 
\begin{multline*}
x^{20} - 4x^{19} + 11x^{18} - 14x^{17} + 27x^{16} - 4x^{15} + 76x^{14} + 26x^{13} + 242x^{12} \\+ 186x^{11} + 1176x^{10} + 1496x^9 + 3207x^8 + 3946x^7 + 4941x^6 \\+ 3726x^5 + 4572x^4 + 2746x^3 + 2191x^2 + 686x + 641
\end{multline*} with LMFDB entry 
$\href{https://www.lmfdb.org/NumberField/20.0.146746349255915802001953125.1}{20.0.146746349255915802001953125.1}$
and its subfields.  For all the fields in this table, we have
$w_{3}=2$, so we have
omitted that column to save space.  We have also truncated the
regulator values; the actual values computed have much higher precision.  As in Example \ref{ex:d4sqrt5}, since this field contains $\QQ(\sqrt{5})$ as a subfield, we expect there to be a discrepancy between $\kappa_3(F_{20})$ and $\size{K_4(\OO_{F_{20}})}$ that includes a nontrivial power of $5$.
\begin{table}[htb]
\caption{The $F_{5}$-extension $F_{20}$ with LMFDB entry \href{https://www.lmfdb.org/NumberField/20.0.146746349255915802001953125.1}{20.0.146746349255915802001953125.1}.}
\label{tab:frob}
\centering
\begin{tabular}{c r r}
\toprule
$F$ & $\text{reg}_3^{\bloch}(F)$ & $\kappa_3(F)$ \\
\midrule
\href{https://www.lmfdb.org/NumberField/1.1.1.1}{$\QQ$} & 0.10017 & $1$ \\ 
\href{https://www.lmfdb.org/NumberField/2.2.5.1}{$\QQ(\sqrt{5})$} & 0.00797 & $5^2$ \\ 
\href{https://www.lmfdb.org/NumberField/4.0.171125.1}{$\QQ(\sqrt{-370+74\sqrt{5}})$} & 0.00797 & $2^{13} \cdot 5^{3} \cdot 113 \cdot 12281$ \\ 
\href{https://www.lmfdb.org/NumberField/5.1.171125.1}{$F_{5}$} & 1149.69233 & $2^3 \cdot 3 \cdot 5$ \\ 
\href{https://www.lmfdb.org/NumberField/10.2.146418828125.1}{$F_{10}$} & 126110.61077 & $2^6 \cdot 3 \cdot 5^6$ \\ 
\href{https://www.lmfdb.org/NumberField/20.0.146746349255915802001953125.1}{$F_{20}$} & 221496113886.93678 & $2^{13} \cdot 3^4 \cdot 5^{11} \cdot 113 \cdot 12281$\\ 
\bottomrule
\end{tabular}
\end{table}
\end{ex}

\begin{ex}\label{ex:d14}
We conclude the examples with a very large one: 
the $D_{14}$-Galois extension $F_{28}$ defined by the polynomial 
\begin{multline*}
x^{28}-7x^{27}+24x^{26}-57x^{25}+109x^{24}-174x^{23}+237x^{22}-293x^{21}\\+328x^{20}-299x^{19}+204x^{18}-102x^{17}+12x^{16}+121x^{15}-351x^{14}+612x^{13}\\-659x^{12}+349x^{11}+66x^{10}-269x^9+268x^8-206x^7+122x^6-17x^5\\-39x^4+24x^3-x^2-3x+1.
\end{multline*}
This field has discriminant 
\[395652986399381320581034818847689 = 3^{14} \cdot 71^{14}.\]  
We give the results of the calculations for it and its subfields in
Table \ref{tab:d14}.  Note that this field does not appear in LMFDB,
although all its proper subfields do.  For $F_{28}$ the Bloch group regulator $\reg_3^{\bloch}(F_{28}) = R$ is
too big to fit in the table; we have 
\[
R=16026302439853319969014775271078.1524626966 \dotsb .
\]

\begin{table}[htb]
\caption{A $D_{14}$-extension (not in LMFDB)}
\label{tab:d14}
\centering
\begin{tabular}{c r r r}
\toprule
$F$ & $\text{reg}_3^{\bloch}(F)$ & $w_3(F)$ & $\kappa_3(F)$ \\
\midrule
\href{https://www.lmfdb.org/NumberField/1.1.1.1}{$\QQ$} & 0.10017 & 2 & $1$ \\
\href{https://www.lmfdb.org/NumberField/2.0.3.1}{$\QQ(\sqrt{-3})$} & 0.10017 & 18 & $1$ \\ 
\href{https://www.lmfdb.org/NumberField/2.0.71.1}{$\QQ(\sqrt{-71})$} & 0.10017 & 2 & $2^3 \cdot 3 \cdot 17$ \\ 
\href{https://www.lmfdb.org/NumberField/2.2.213.1}{$\QQ(\sqrt{213})$} & 2431.72430 & 2 & $1$ \\ 
\href{https://www.lmfdb.org/NumberField/4.0.45369.1}{$\QQ(\sqrt{-3}, \sqrt{-71})$} & 1215.86215 & 18 & $2^4 \cdot 3 \cdot 17$ \\ 
\href{https://www.lmfdb.org/NumberField/7.1.357911.1}{$F_7$} & 243.25718 & 2 & $1$ \\ 
\href{https://www.lmfdb.org/NumberField/14.0.9095120158391.1}{$F_{14,a}$} & 590728.02894 & 2 & $2^3 \cdot 3 \cdot 17$ \\ 
\href{https://www.lmfdb.org/NumberField/14.0.280155320935227.1}{$F_{14,b}$} & 11500541184526.08498 & 18 & $1$ \\ 
\href{https://www.lmfdb.org/NumberField/14.2.19891027786401117.1}{$F_{14,c}$} & 279182912219085362.08643 & 2 & $1$ \\ 
$F_{28}$ & $R$ & 18 & $2^4 \cdot 3 \cdot 17$ \\ 
\bottomrule
\end{tabular}
\end{table}

Figure \ref{fig:field_diagram} shows the lattice of subfields of
$F_{28}$ together with the quantities $(w_{3}, \kappa_{3})$.  The
times symbol $\times$ is used to abbreviate multiple isomorphic copies
of the same field.  Thus for instance $F_{7}\times 7$ indicates that
there are in fact $7$ isomorphic copies of the field labeled $F_{7}$
appearing as subfields.  The $\times 7$ on the poset edge indicates
that all seven of these appear as subfields of $F_{14,a}$. 
 Note that from this picture we can easily see that whenever $E\subset F$, we have $w_3(F)\mid w_3(E)$, and also that $\kappa_3(F)\mid \kappa_3(E)$ as in Question \ref{qq}.

\begin{figure}[htb]
  \begin{tikzpicture}[auto, node distance=2cm]
    \node (e) [align=center] {\href{https://www.lmfdb.org/NumberField/2.0.3.1}{$\QQ$}\\$(2,1)$};
    \node[above of=e] (2neg71) [align=center]{\href{https://www.lmfdb.org/NumberField/2.0.71.1}{$\QQ(\sqrt{-71})$}\\$(2,408)$};
    \node[left of=2neg71, node distance=3cm] (2neg3) [align=center] {\href{https://www.lmfdb.org/NumberField/2.0.3.1}{$\QQ(\sqrt{-3})$}\\$(18,1)$};
    \node[right of=2neg71, node distance=3cm] (2pos213) [align=center]{\href{https://www.lmfdb.org/NumberField/2.2.213.1}{$\QQ(\sqrt{213})$}\\$(2,1)$};
    \node at (0.75,4) (4) [align=center]{\href{https://www.lmfdb.org/NumberField/4.0.45369.1}{$\QQ(\sqrt{-3}, \sqrt{-71})$}\\$(18,816)$};
   \node at (-3,6) (7) [align=center]{\href{https://www.lmfdb.org/NumberField/7.1.357911.1}{$F_7 \times 7$}\\$(2,1)$};
    \node[above of=7, node distance=4cm] (14a) [align=center]{\href{https://www.lmfdb.org/NumberField/14.0.9095120158391.1}{$F_{14,a}$}\\$(2,408)$};
    \node[left of=14a, node distance=3cm] (14b) [align=center]{\href{https://www.lmfdb.org/NumberField/14.0.280155320935227.1}{$F_{14,b} \times 7$}\\$(18,1)$};
    \node[right of=14a, node distance=3cm] (14c) [align=center]{\href{https://www.lmfdb.org/NumberField/14.2.19891027786401117.1}{$F_{14,c} \times 7$}\\$(2,1)$};
    \node[above of=14a] (28) [align=center] {$F_{28}$\\$(18,816)$};
    \draw[-] (e) to (2neg3);
    \draw[-] (e) to (2neg71);
    \draw[-] (e) to (2pos213);
    \draw[-] (e) to (7);
    \draw[-] (4) to (2neg3);
    \draw[-] (4) to (2neg71);
    \draw[-] (4) to (2pos213);
    \draw[-] (4) to (28);
    \draw[-] (14a) to node[left] {$\times 7$} (7);
    \draw[-] (14a) to (2neg71);
    \draw[-] (14b) to (2neg3);
    \draw[-] (14c) to (2pos213);
    \draw[-] (14b) to (7);
    \draw[-] (14c) to (7);
    \draw[-] (14a) to (28);
    \draw[-] (14b) to (28);
    \draw[-] (14c) to (28);
  \end{tikzpicture}
\caption{The lattice of subfields of
$F_{28}$ together with the quantities $(w_{3}, \kappa_{3})$.  The
times symbol $\times$ is used to abbreviate multiple isomorphic copies
of the same field, or omitted multiple edges in the poset.}\label{fig:field_diagram}
\end{figure}
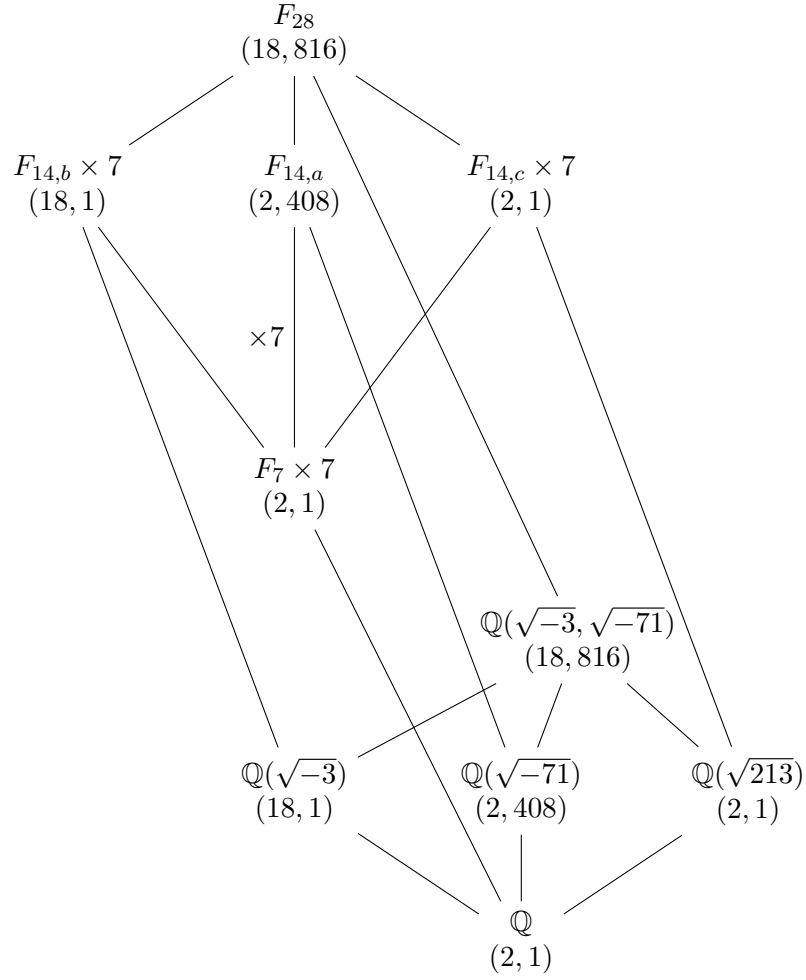

\end{ex}

\clearpage

\section{The cyclotomic case}\label{s:cyclo}
In the cyclotomic case $F=\QQ(\zeta_n)$, we encounter factors of $\kappa_3(F)$ which are not part of $K_4(\OO_F)$.  For example, the work of Rognes--\O stv\ae r \cite{rognes-ostaer} shows that the order of $K_4$ for the ring of integers of $\QQ(\zeta_{2^k})$ and for its maximal totally real subfield are odd, but our calculations have $\kappa_3$ values that are even.  See Table~\ref{tab:cyclotomic}.  
In our calculations for $n=p$ prime, we have  factors $p^{(p-1)/2}$ in $\kappa_3(F)$ which we know are not part of $K_4(\OO_F)$.  If $p$ is an odd regular prime, then $K_4(\ZZ[\zeta_p])$ has no $p$-torsion  \cite[p.~170, Example~75]{weibel-handbook5}.  This includes all odd primes $p < 100$ except $37$, $59$, and $67$.  Furthermore, in the associated totally real subfield of $\QQ(\zeta_p)$ in this range, we still encounter this factor of $p^{(p-1)/2}$ in $\kappa_3$.   

\begin{table}[htb]
\caption{Comparing $\kappa_3$ for some cyclotomic fields $F = \QQ(\zeta_n)$ and their totally real subfields $F^+ = \QQ(\zeta_n + \zeta_n^{-1})$} 
\label{tab:cyclotomic}
\centering
\setlength{\tabcolsep}{7pt} 
\begin{tabular}{r r r r}
\toprule
$n$ & $\abs{\Delta_{F}}$ & $\kappa_3(F)$ & $\kappa_3(F^+)$\\
\midrule
\href{https://www.lmfdb.org/NumberField/2.0.3.1}{$3,6$} & $3$ & $1$ & $1$\\
\href{https://www.lmfdb.org/NumberField/2.0.4.1}{$4$} & $2^2$ & $1$ & $1$\\
\href{https://www.lmfdb.org/NumberField/4.0.125.1}{$5,10$} & $5^3$ & $5^2$ & $5^2$ \\
\href{https://www.lmfdb.org/NumberField/6.0.16807.1}{$7, 14$} & $7^5$ & $2^3 \cdot 7^3$ & $7^3$\\
\href{https://www.lmfdb.org/NumberField/4.0.256.1}{$8$} & $2^8$ & $2^4 \cdot 3$ & $2^2$\\
\href{https://www.lmfdb.org/NumberField/6.0.19683.1}{$9, 18$} & $3^9$ & $3^3 \cdot 13$ & $3^3$\\
\href{https://www.lmfdb.org/NumberField/10.0.2357947691.1}{$11, 22$} & $11^9$ & $3 \cdot 11^5 \cdot 1301$ & $11^5$\\ 
\href{https://www.lmfdb.org/NumberField/4.0.144.1}{$12$} & $2^4 \cdot 3^2$ & $2^2$ & $1$\\
\href{https://www.lmfdb.org/NumberField/12.0.1792160394037.1}{$13, 26$} & $13^{11}$ & $3^5 \cdot 13^6 \cdot 37 \cdot 61$ & $3 \cdot 13^6$ \\ 
\href{https://www.lmfdb.org/NumberField/8.0.1265625.1}{$15, 30$} & $3^4 \cdot 5^6$ & $2^4\cdot 5^4$ & $5^4$\\
\href{https://www.lmfdb.org/NumberField/8.0.16777216.1}{$16$} & $2^{24}$ & $2^{12} \cdot 3 \cdot 137$ & $2^8$\\
\href{https://www.lmfdb.org/NumberField/16.0.2862423051509815793.1}{$17,34$} & $17^{15}$ & $2^4 \cdot 17^8 \cdot 18913 \cdot 32609$ & $2^3 \cdot 17^8$ \\
\href{https://www.lmfdb.org/NumberField/18.0.5480386857784802185939.1}{$19, 38$} & $19^{17}$ & $11 \cdot 19^9 \cdot 1297 \cdot 3187 \cdot 3559$ & $19^9$ \\
\href{https://www.lmfdb.org/NumberField/8.0.4000000.1}{$20$} & $2^8 \cdot 5^6$ & $2^4 \cdot 3 \cdot 5^5$ & $5^4$ \\ 
\href{https://www.lmfdb.org/NumberField/12.0.205924456521.1}{$21, 42$} & $3^6 \cdot 7^{10}$ & $2^4 \cdot 7^7 \cdot 43$ & $7^6$ \\ 
\href{https://www.lmfdb.org/NumberField/22.0.39471584120695485887249589623.1}{$23, 46$} & $23^{21}$ & $2^3 \cdot 3 \cdot 23^{11} \cdot 251461857575303$ & $23^{11}$ \\ 
\href{https://www.lmfdb.org/NumberField/8.0.5308416.1}{$24$} & $2^{16} \cdot 3^4$ & $2^8 \cdot 3 \cdot 23$ & $2^4$ \\
\href{https://www.lmfdb.org/NumberField/20.0.2910383045673370361328125.1}{$25, 50$} & $5^{35}$ & $5^{20} \cdot 41 \cdot 61 \cdot 2141 \cdot 154081$ & $5^{18}$ \\ 
\href{https://www.lmfdb.org/NumberField/18.0.2954312706550833698643.1}{$27, 54$} & $3^{45}$ & $3^{18} \cdot 13 \cdot 19 \cdot 2467 \cdot 52057$ & $3^{15}$ \\ 
\href{https://www.lmfdb.org/NumberField/12.0.1157018619904.1}{$28$} & $2^{12} \cdot 7^{10}$ & $2^9 \cdot 3 \cdot 7^6 \cdot 1153$ & $3 \cdot 7^6$\\
\href{https://www.lmfdb.org/NumberField/28.0.3053134545970524535745336759489912159909.1}{$29, 58$}  & $29^{27}$ & $2^3 \cdot 29^{14} \cdot 1297 \cdot 1373 \cdot  32369 \cdot 471402989989$ & $29^{14}$  \\
\bottomrule
\end{tabular}
\end{table}

The discrepancy between $\kappa_3$ of a cyclotomic field and its maximal totally real subfield is related to generalized Bernoulli numbers $B_{3,\chi}$ for Dirichlet characters $\chi$.  An analogous calculation investigating the ratio of orders of $K_4$ for a CM extension and its totally real subfield has been done by Voight \cite{Voight}, where this ratio is higher relative  class number. The Dedekind zeta function of a cyclotomic field $\QQ(\zeta_n)$ is the product of $L$-functions for the Dirichlet characters of modulus $n$.  The ratio of Dedekind zeta functions for the cyclotomic field and its maximal totally real subfield is the product of $L$-functions for the odd Dirichlet characters of modulus $n$.  Thus from \eqref{eq:kappa3F}, we get 
\[\frac{\kappa_3(\QQ(\zeta_n))}{\kappa_3(\QQ(\zeta_n)^+)} = \frac{1}{(-2)^{\phi(n)/2}} \left(\prod_{\text{$\chi$ odd}} L(-2,\chi) \right) \frac{\reg^{\bloch}_3(\QQ(\zeta_n)^+)}{\reg^{\bloch}_3(\QQ(\zeta_n))} \frac{w_3(\QQ(\zeta_n))}{w_3(\QQ(\zeta_n)^+)},\]
where $\phi(n) = \deg(\QQ(\zeta_n))$.  The generalized Bernoulli number is $B_{3,\chi} = -3L(-2,\chi)$. 
Thus 
\begin{equation}\label{eq:kapparatio}
\frac{\kappa_3(\QQ(\zeta_n))}{\kappa_3(\QQ(\zeta_n)^+)} = \frac{1}{(-2\cdot 3)^{\phi(n)/2}} \left(\prod_{\text{$\chi$ odd}} B_{3,\chi} \right) \frac{\reg^{\bloch}_3(\QQ(\zeta_n)^+)}{\reg^{\bloch}_3(\QQ(\zeta_n))} \frac{w_3(\QQ(\zeta_n))}{w_3(\QQ(\zeta_n)^+)}.    
\end{equation}

Since the cyclotomic field is a CM extension of its maximal totally real subfield, the lattices $\Lambda_3(\QQ(\zeta_n))$ and $\Lambda_3(\QQ(\zeta_n)^+)$ have the same rank, and we do not expect the index to be large. Furthermore,  $w_3(\QQ(\zeta_n)^+) = 2$ by Proposition~\ref{prop:wmF2}, and $w_3(\QQ(\zeta_n)) = \lcm(2,n)$ or $3\lcm(2,n)$.  
Thus we expect any large prime discrepancy between $\kappa_3(\QQ(\zeta_n))$ and $\kappa_3(\QQ(\zeta_n)^+)$ to come from the numerators of generalized Bernoulli numbers $B_{3,\chi}$.  See Table~\ref{tab:cyclotomic-reg-ratio}.
\begin{table}[htb]
\caption{The product the third and fourth columns divided by the second column is typically 1, and, if not, only involves a small product of small primes.  Thus any large primes appearing in the ratio of $\kappa_3$'s in \eqref{eq:kapparatio} must come from the numerator of the product of generalized Bernoulli numbers.} 
\label{tab:cyclotomic-reg-ratio}
\centering
\setlength{\tabcolsep}{7pt} 
\begin{tabular}{r r r r r}
\toprule
$n$ & denom $\frac{1}{(-2\cdot 3)^{\phi(n)/2}} \left(\prod_{\text{$\chi$ odd}} B_{3,\chi} \right)$ & $\frac{\reg^{\bloch}_3(\QQ(\zeta_n)^+)}{\reg^{\bloch}_3(\QQ(\zeta_n))}$ & $\frac{w_3(\QQ(\zeta_n))}{w_3(\QQ(\zeta_n)^+)}$\\
\midrule
\href{https://www.lmfdb.org/NumberField/2.0.3.1}{$3,6$} & $3^2$ & $1$ & $3^2$\\
\href{https://www.lmfdb.org/NumberField/2.0.4.1}{$4$} &  $2^2$ & $2$ & $2$\\
\href{https://www.lmfdb.org/NumberField/4.0.125.1}{$5,10$} &  $5$ & $1$ & $5$ \\
\href{https://www.lmfdb.org/NumberField/6.0.16807.1}{$7, 14$} & $7$ & $1$ & $7$\\
\href{https://www.lmfdb.org/NumberField/4.0.256.1}{$8$} & $2^3$ & $2^3$ & $2^2$\\
\href{https://www.lmfdb.org/NumberField/6.0.19683.1}{$9, 18$} & $3^3$ & $1$ & $3^3$\\
\href{https://www.lmfdb.org/NumberField/10.0.2357947691.1}{$11, 22$} & $11$ & $1$ & $11$\\\
\href{https://www.lmfdb.org/NumberField/4.0.144.1}{$12$} & $2^2 \cdot 3^2$ & $2^3$ & $2 \cdot 3^2$\\
\href{https://www.lmfdb.org/NumberField/12.0.1792160394037.1}{$13, 26$} & $13$ &  $1$ & $13$\\\ 
\href{https://www.lmfdb.org/NumberField/8.0.1265625.1}{$15, 30$} & $3^2 \cdot 5$ &  $2$ & $3^2 \cdot 5$\\
\href{https://www.lmfdb.org/NumberField/8.0.16777216.1}{$16$} & $2^4$ & $2^5$ & $2^3$\\
\href{https://www.lmfdb.org/NumberField/16.0.2862423051509815793.1}{$17,34$} & $17$ & $1$ & $17$\\
\href{https://www.lmfdb.org/NumberField/18.0.5480386857784802185939.1}{$19, 38$} & $19$ & $1$ & $19$\\
\href{https://www.lmfdb.org/NumberField/8.0.4000000.1}{$20$} & $2^2$ & $2^5$ & $2 \cdot 5$\\ 
\href{https://www.lmfdb.org/NumberField/12.0.205924456521.1}{$21, 42$} & $3^2$ & $2$ & $3^2 \cdot 7$\\ 
\href{https://www.lmfdb.org/NumberField/22.0.39471584120695485887249589623.1}{$23, 46$} & $23$ & $1$ & $23$\\ 
\href{https://www.lmfdb.org/NumberField/8.0.5308416.1}{$24$} & $2^3 \cdot 3$ & $2^5$ & $2^2 \cdot 3^2$\\
\href{https://www.lmfdb.org/NumberField/20.0.2910383045673370361328125.1}{$25, 50$} & $5^2$ &  $5^2$ & $5^2$\\ 
\href{https://www.lmfdb.org/NumberField/18.0.2954312706550833698643.1}{$27, 54$} & $3^4$ & $3^3$ & $3^4$\\ 
\href{https://www.lmfdb.org/NumberField/12.0.1157018619904.1}{$28$} & $7$ & $2^7$ & $2 \cdot 7$\\
\href{https://www.lmfdb.org/NumberField/28.0.3053134545970524535745336759489912159909.1}{$29, 58$}  &  $29$ & $1$ & $29$\\
\bottomrule
\end{tabular}

\end{table}

\section{\texorpdfstring{$\zeta$}{zeta}-lyzing}\label{s:zl}
As it happens, the setting developed above can be turned into a practical advantage. It improves our ability to compute $\zeta_F(3)$ for number fields $F$ where existing algorithms struggle to deliver more than a few dozens of accurate digits.

Take, for instance, the degree $14$ field \href{https://www.lmfdb.org/NumberField/14.14.19852344578500703125.1}{14.14.19852344578500703125.1}, whose Galois closure $14T61$ is substantial, of degree $2\cdot 7!$. Computing the Dedekind $\zeta$-value at $3$ to high precision is notably difficult for this field. gp/PARI, at $100$-bit realprecision, returns
\[
1.000146684702559704405958901,
\]
and yet only the first $19$ of the $28$ digits produced are actually correct.

By contrast, the regulator can be rapidly computed to arbitrary precision using the single-valued trilogarithm $\tCL_3(z)$ evaluated on a non-torsion element of $\wB_3(F)$. 

The weight $3$ Bloch group regulator comes out as  
\[
\reg^{\bloch}_3(F) \ceq 27358089271734794708.291653037194757361068384\dots\]
Using the definition in Conjecture~\ref{conj:kappa}, 
we obtain an approximation to the expected integer $\kappa_3(F)$ as
\[6103515625.000000000223.\]
The nine zeros after the decimal point strongly suggest that the true value is the rounded integer $6103515625 = 5^{14}$.

This lets us turn the tables.  Since polylogarithms admit arbitrary precision computations, we can also take advantage of them to pin down $\zeta_F(3)$ itself to arbitrary precision. In effect, the few reliable digits obtainable by current algorithms serve as leverage---a seed that we then catalyze into full precision, providing justification to call the process ``$\zeta$-lyzing''.

\bibliographystyle{amsplain_initials_eprint}
\bibliography{weight_3}

 \end{document}